\documentclass[letterpaper,11pt]{article}
\usepackage{color}
\usepackage{pdflscape}

\usepackage{tabularx}
\usepackage{graphicx}
\usepackage{adjustbox}

\usepackage[small]{titlesec}
\usepackage{theorem,amsmath,amssymb,amscd}
\usepackage[all,cmtip]{xy}
\usepackage{booktabs}
\usepackage[hyperfootnotes=false, colorlinks, linkcolor={blue}, citecolor={magenta}, filecolor={blue}, urlcolor={blue}]{hyperref}
\usepackage[overload]{textcase} 

\theoremstyle{change}
\allowdisplaybreaks
\newcommand{\Q}{{\mathbb Q}}
\newcommand{\Z}{{\mathbb Z}}

\newcommand{\C}{{\mathbb C}}

\newcommand{\p}{\mathfrak p}
\newcommand{\OF}{{\mathfrak o}}
\newcommand{\Fq}{{\mathbb F}_q}
\newcommand{\GL}{{\rm GL}}

\newcommand{\SL}{{\rm SL}}

\newcommand{\GSp}{{\rm GSp}}

\newcommand{\Sp}{{\rm Sp}}

\newcommand{\St}{{\rm St}}

\newcommand{\Si}[1]{{\rm Si}(\p^{#1})}

\newcommand{\cInd}{\text{\rm c-Ind}}

\newcommand{\qed}{\hspace*{\fill}\rule{1ex}{1ex}}
\newcommand{\forget}[1]{}
\def\qdots{\mathinner{\mkern1mu\raise0pt\vbox{\kern7pt\hbox{.}}\mkern2mu
\raise3.4pt\hbox{.}\mkern2mu\raise7pt\hbox{.}\mkern1mu}}

\newenvironment{proof}{\vspace{1ex}\noindent\emph{Proof.}\hspace{0.5em}}
	{\hfill\qed\vspace{2ex}}
\newenvironment{bsmallmatrix}{\left[\begin{smallmatrix}}{\end{smallmatrix}\right]}

\newtheorem{lemma}{Lemma.}[section]
\newtheorem{theorem}[lemma]{Theorem.}
\newtheorem{corollary}[lemma]{Corollary.}

\begin{document}

\thispagestyle{empty}

\begin{center}
 {\bf\Large Siegel $\mathfrak{p}^2$ Vectors for Representations of $\GSp(4)$}

 \vspace{3ex}
 Jonathan Cohen
 
 \vspace{3ex}
 \begin{minipage}{80ex}
  \small\textbf{Abstract.} Let $F$ be a $p$-adic field and $(\pi, V)$ an irreducible complex representation of $G=\GSp(4, F)$ with trivial central character. Let $\Si{2}\subset G$ denote the Siegel congruence subgroup of level $\p^2$ and $u\in N_G(\Si{2})$ the Atkin-Lehner element. We compute the dimension of the space of $\Si{2}$-fixed vectors in $V$ as well as the signatures of the involutions $\pi(u)$ acting on these spaces. 
 \end{minipage}
 \vspace{3ex}
\end{center}

\section{Introduction}

This paper is concerned with a dimension counting problem in $p$-adic representation theory for the group $\GSp(4)$. Part of the motivation comes from the following problem in the classical theory of Siegel modular forms. If $\Gamma\subset \Sp(4, \Q)$ is a congruence subgroup, then the dimension of the space of cusp forms of level $\Gamma$ (and integer weight $k$, say) is not known in general. For example, if $$\Gamma=\begin{bmatrix}
\Z & 4\Z & \Z & \Z \\
\Z & \Z & \Z & \Z \\
\Z & 4\Z & \Z & \Z \\
4\Z & 4\Z & 4\Z & \Z 
\end{bmatrix}\cap \Sp(4, \Z)$$ is the ``Klingen congruence subgroup of level $4$'' then the associated dimensions were only recently computed in \cite{RSY2022}. The method by which this was achieved required, as one of its several inputs, the  dimensions of spaces of fixed vectors in all irreducible smooth representations of $\GSp(4, \Q_2)$ for the subgroups $${\rm Kl}(4)=\begin{bmatrix}
\Z_2 & \Z_2 & \Z_2 & \Z_2 \\
4\Z_2 & \Z_2 & \Z_2 & \Z_2 \\
4\Z_2 & \Z_2 & \Z_2 & \Z_2 \\
4\Z_2 & 4\Z_2 & 4\Z_2 & \Z 
\end{bmatrix}\cap \GSp(4, \Z_2).$$ These dimensions, and more, had been computed in \cite{Yi2021}. We remark that the different ``shape'' of the subgroups in question is an artifact of different conventions for alternating forms in the classical and representation-theoretic contexts. 

If one hopes to use the approach taken in \cite{RSY2022} for other congruence subgroups, then a necessary component is the determination of dimensions of spaces of fixed vectors in all irreducible smooth representations $(\pi, V)$ of $\GSp(4, \Q_p)$ for appropriate local subgroups. In this paper we carry out this local computation for the ``Siegel congruence subgroups of level $p^2$'' given by $${\rm Si}(p^2)=\begin{bmatrix}
\Z_p & \Z_p & \Z_p &  \Z_p \\
\Z_p & \Z_p & \Z_p &  \Z_p \\
p^2\Z_p & p^2\Z_p & \Z_p &  \Z_p \\
p^2\Z_p & p^2\Z_p & \Z_p &  \Z_p \\
\end{bmatrix}\cap \GSp(4, \Z_p).$$ If we let $$u_2:=\begin{bmatrix}
 & & 1 & \\
  && & -1\\
  p^2  & & & \\
   & -p^2 & & 
\end{bmatrix}\in N_G({\rm Si}(p^2))$$ then $u_2$ acts by an involution on the space of ${\rm Si}(p^2)$-fixed vectors in any representation of $\GSp(4, \Q_p)$ with trivial central character. We compute the signature of this action. Our main result is stated in Theorem \ref{Main theorem}. We remark that the case $p=2$ may be found in \cite{RSY2022}, and the case where $(\pi, V)$ has Iwahori-invariant vectors 
is handled in \cite{Ro2006}. 

 Our situation is in one sense easier than the corresponding Klingen subgroup because ${\rm Si}(p^2)$ contains a conjugate of the pro-unipotent radical of $\GSp(4, \Z_p)$, and the work of \cite{Ros2018} then permits us to reduce to a finite computation employing the character tables of \cite{Sh1982} and \cite{En1972}. Since there are some errors in these character tables, we include a fully corrected table (when $p=2$) since there does not appear to be one published elsewhere; see tables \ref{tab:table8} and \ref{tab:table9}. When $q$ is odd the only errors we have found in \cite{Sh1982} are that the three entries in the $\tau_2$ column for the conjugacy classes of type $C$ should all be negated. 
 
 We thank Ralf Schmidt for many useful conversations in preparing this article. 

\section{Notation}\label{notation section}
Let $\Fq$ be the field with $q$ elements, where $q$ is a power of a prime number $p$. If $q$ is odd let $\xi $ denote a fixed non-square in $\Fq^\times$, and $\alpha_0:\Fq^\times \to \{\pm 1\} $ the nontrivial quadratic character.  

 Let $F$ be a field. Throughout we fix the symplectic form
\begin{equation}\label{Jdefeq}
 J=\begin{bsmallmatrix}&&&1\\&&1\\&-1\\-1\end{bsmallmatrix}.
\end{equation}
We let $\GSp(4)$ be the algebraic group whose $F$-points are
\begin{equation}\label{GSp4defeq}
 \GSp(4,F)=\{g\in\GL(4,F)\mid\ ^tgJg=\mu(g)J\text{ for some }\mu(g)\in F^\times\}.
\end{equation}
The kernel of the multiplier homomorphism $\mu:\GSp(4,F)\to F^\times$ is the symplectic group $\Sp(4,F)$. The center of $\GSp(4,F)$ is
\begin{equation}
 Z(F)=\{\begin{bsmallmatrix}a\\&a\\&&a\\&&&a\end{bsmallmatrix}\mid a\in F^\times\}
\end{equation}
If $F$ is a non-archimedean local field of characteristic zero, we let $\OF$ be its ring of integers, $\p$~the maximal ideal of~$\OF$, and $\varpi$ a generator of $\p$. We denote by~$q$ the cardinality of the residue class field~$\OF/\p$. We further define the maximal compact subgroup 
\begin{eqnarray}
K = \GSp(4, \OF),
\end{eqnarray} its prounipotent radical 
\begin{eqnarray}
K^+ = \ker(K\to \GSp(4, \OF/\p)),
\end{eqnarray}
and for $n\geq 1$, the Siegel congruence subgroups 
\begin{equation}\label{Si(n)}
 \Si{n}=K\cap\begin{bsmallmatrix}  \OF & \OF&\OF&\OF\\ \OF & \OF&\OF&\OF\\
 \p^n & \p^n&\OF&\OF\\
 \p^n & \p^n&\OF&\OF\end{bsmallmatrix}. 
\end{equation} We also define the Atkin-Lehner element $u_n\in N_G(\Si{n})$ by 
\begin{equation}
u_n = \begin{bsmallmatrix}
& & 1 & \\
 & & & -1 \\
 \varpi^n & &  & \\
 & -\varpi^n & & 
\end{bsmallmatrix}. 
\end{equation} Observe that $u_n^2 = \varpi^n I_4$. 

Let $(\pi, V)$ be a smooth complex representation of $G:=\GSp(4, F)$ where $F$ is a $p$-adic field. For a compact open subgroup $C\subset G$ we define \begin{equation}
V^C : = \{v\in V \mid \ \pi(c)v=v  \text{ for all } c\in C \}. 
\end{equation} For an irreducible representation $\rho$ of any group, we write $\omega_\rho$ for its central character.

\section{Main result and proof}

We now state and prove the main result of this article. 

\begin{theorem}\label{Main theorem}
Let $(\pi, V)$ be an irreducible smooth complex representation of $G$ with trivial central character.  The integers $\dim V^{\Si{2}}$ and the signatures of $\pi(u_2)$ are given in Table \ref{tab:table4} when $q$ is even and in tables \ref{tab:table6} and \ref{tab:table7} when $q$ is odd. 


\end{theorem}


\begin{proof}
Define the two groups $H$, $M$ and the element $u$ by 
\begin{eqnarray}
H&=& K\cap\begin{bsmallmatrix}  \OF & \OF&\p&\p\\ \OF & \OF&\p&\p\\
 \p & \p&\OF&\OF\\
 \p & \p&\OF&\OF\end{bsmallmatrix}  = \begin{bsmallmatrix}
 \varpi & & & \\
  & \varpi& & \\
   &   & 1 & \\
    & & & 1
 \end{bsmallmatrix}\Si{2}\begin{bsmallmatrix}
  \varpi^{-1} & & & \\
   & \varpi^{-1}& & \\
    &   & 1 & \\
     & & & 1
  \end{bsmallmatrix}\\
 M&=& H/K^+ = \left\{\begin{bsmallmatrix}
a & b & & \\
c & d & & \\
& & \lambda a & -\lambda b \\
 & & -\lambda c & \lambda d
  \end{bsmallmatrix} \mid \lambda, ad-bc \in \Fq^\times   \right\} \\
   u &=& \begin{bsmallmatrix} & & 1 & \\
   & & & -1 \\
    1 & & & \\
     & -1 & &  
   \end{bsmallmatrix}  = \varpi^{-1} I_4 \begin{bsmallmatrix}
    \varpi & & & \\
     & \varpi& & \\
      &   & 1 & \\
       & & & 1
    \end{bsmallmatrix}u_2  \begin{bsmallmatrix}
     \varpi^{-1} & & & \\
      & \varpi^{-1} & & \\
       &   & 1 & \\
        & & & 1
     \end{bsmallmatrix}\in N_{G}(H) 
\end{eqnarray}  Observe that 
$M\cong \GL(2, \Fq)\times \Fq^\times$. If $q$ is even then $\GSp(4, \Fq) = Z(\Fq)\times \Sp(4, \Fq)$ and $M = Z(\Fq)\times (M\cap \Sp(4, \Fq))$. Note $u\in \Sp(4, F)\cap K$. Let $s(\pi)$ denote the signature of $\pi(u_2)$ acting on $V^{\Si{2}}$, or equivalently, of $\pi(u)$ acting on $V^H$. Let $\chi$ denote the trace character of $V^{K^+}$ as a $K/K^+=\GSp(4, \Fq)$-representation. We have  \begin{eqnarray}
\dim V^{\Si{2}} &= &\dim V^H \nonumber\\
&= &\dim ((V^{K^+})^M) \nonumber \\
  &=& |M|^{-1}\sum\limits_{m\in M} \chi(m) \\
  s( \pi ) &=&  |M|^{-1}\sum\limits_{m\in M}\chi(mu)
\end{eqnarray}  where we identify $u$ with its image in $K/K^+$.

For each $(\pi, V)$, the corresponding $\chi$ is given in table 3 of \cite{Ros2018} in terms of the virtual characters listed in \cite{Sh1982} and \cite{En1972}. In tables \ref{tab:table8} and \ref{tab:table9} are corrected versions of the latter character tables.

{\bf Proof of dimension formulas.} We must determine the distribution of conjugacy classes of $ \GSp(4, \Fq)$ in $M$.

  \tabcolsep=0.11cm
 \begin{table}[h!]
   \begin{center}
     \caption{The representatives of the conjugacy classes of $M$, with the sizes of the $M$-conjugacy class and the labels of the corresponding classes in $\GSp(4, \Fq)$, $q$ odd, as listed in \cite{Sh1982}. Some rows have representative given in a quadratic extension of $\Fq$. In the last column we also indicate the equivalences that occur over $M$, for purposes of counting multiplicities.}
     \label{tab:table1}
     \begin{tabular}{c|c|c|c} 
    $g$ & 
        $[M:C_M(g)]$ & conditions &  $K/K^+$-conjugacy class  \\
       \hline
       
       $\begin{bsmallmatrix}
          a & & & \\
           & a & & \\
           & & a & \\
            & & & a 
          \end{bsmallmatrix}$  & $1$      &  $a\in \Fq^\times$ &   $A_0(a)$    \\

      $\begin{bsmallmatrix}
               a & 1 & & \\
                & a & & \\
                & & a & 1 \\
                 & & & a 
               \end{bsmallmatrix}$ & $q^2-1$ & $a\in \Fq^\times$  &  $A_{21}(a)$\\
               
          $\begin{bsmallmatrix}
                       a &  & & \\
                        & -a & & \\
                        & & -a &  \\
                         & & & a 
                       \end{bsmallmatrix}$ & $q^2+q$ & $a\in \Fq^\times$ &  $B_0(a)=B_0(-a)$\\  
                       
       $\begin{bsmallmatrix}
                            a &  & & \\
                             & -a & & \\
                             & & -a &  \\
                              & & & a 
                            \end{bsmallmatrix}$ & $q^2-q$ & $a^{q}=-a$  &  $C_0(a) = C_0(-a)$\\

     $\begin{bsmallmatrix}
     a & & & \\
      & a & & \\
      & & b & \\
       & & & b 
     \end{bsmallmatrix}$ & $1$ & $a,b\in \Fq^\times$, $a\neq b$  & $D_0(a,b)$ \\
     
      $\begin{bsmallmatrix}
        a & & & \\
         & b & & \\
         & & a & \\
          & & & b 
        \end{bsmallmatrix}$ & $q^2+q$ & $a,b\in \Fq^\times$, $a\neq b$ & $D_0(a,b)=D_0(b,a)$ \\

     $\begin{bsmallmatrix}
          a & 1 & & \\
           & a & & \\
           & & b & 1 \\
            & & & b 
          \end{bsmallmatrix}$ & $q^2-1$ & $a,b\in \Fq^\times$, $a\neq b$ &  $D_1(a,b)$\\

            $\begin{bsmallmatrix}
                                          ac &  & & \\
                                           & c & & \\
                                           & & c &  \\
                                            & & & ca^{-1}
                                          \end{bsmallmatrix}$ & $q^2+q$ & $a,c\in \Fq^\times$, $a\neq \pm 1$  &  $E_0(c,a)$\\

    $\begin{bsmallmatrix}
          a & & & \\
           & a^q & & \\
           & & a & \\
            & & & a^q 
          \end{bsmallmatrix}$ & $q^2-q$ & $a^q\neq a$ & $F_0(a) = F_0(a^q)$ \\

           $\begin{bsmallmatrix}
                                                    a_1 &  & & \\
                                                     & a_2 & & \\
                                                     & & a_3 &  \\
                                                      & & & a_4
                                                    \end{bsmallmatrix}$ & $q^2+q$ & $\substack{  a_i\in \Fq^\times,\  a_1a_4=a_2a_3, \\ a_i\neq a_j \text{ if } i\neq j} $ &  $H_0(a_1,a_2, a_3, a_4) =H_0(a_2, a_1, a_4, a_3)$\\  
                                                    
              $\begin{bsmallmatrix}
                       a & & & \\
                        & a^q & & \\
                        & & ba^{-q} & \\
                         & & & ba^{-1}
                       \end{bsmallmatrix}$ & $q^2-q$ & $\substack{  a^q\neq a, \ b\in \Fq^\times\\ b\neq a^2, a^{q+1}}$ & $I_0(a,b) = I_0(a^q, b)$ \\

     \end{tabular}
   \end{center}
 \end{table}
  \clearpage

  \tabcolsep=0.11cm
  \begin{table}[h!]
    \begin{center}
      \caption{The first eight rows give the conjugacy classes of $M\cap \Sp(4, \Fq)$ with the sizes of the $M$-conjugacy class and the labels of the corresponding classes in $\Sp(4, \Fq)$, $q$ even, as listed in \cite{En1972}. Some elements have representatives given in a quadratic extension of $\Fq$. In particular $\theta$ denotes an element of order $q^2-1$ in $\mathbb{F}_{q^2}$, with $\eta=\theta^{q-1}$ and $\gamma = \theta^{q+1}$. The last four rows give representatives of the $\Sp(4, q)$-conjugacy classes of $Mu\cap \Sp(4, q)$.  In the last column we also indicate the equivalences that occur over $M$, for purposes of counting multiplicities.}
      \label{tab:table2}
      \begin{tabular}{c|c|c|c} 
     $g$ & 
         $[M:C_M(g)]$ & conditions &  $K/K^+$-conjugacy class  \\
        \hline
        
        $\begin{bsmallmatrix}
           1 & & & \\
            & 1 & & \\
            & & 1 & \\
             & & & 1 
           \end{bsmallmatrix}$  & $1$      & &   $A_1$    \\

       $\begin{bsmallmatrix}
                1 & 1 & & \\
                 & 1 & & \\
                 & & 1 & 1 \\
                  & & & 1 
                \end{bsmallmatrix} $ & $q^2-1$ &   &  $A_{31}$\\

           $\begin{bsmallmatrix}
                        \gamma^i &  & & \\
                         & \gamma^j & & \\
                         & & \gamma^{-j} &  \\
                          & & & \gamma^i 
                        \end{bsmallmatrix}$ & $q^2+q$ & $\gamma^i$, $\gamma^j$, $\gamma^{i\pm j}\neq 1$ &  $B_1(i,j)=B_1(j,i)$\\  
                        
        $\begin{bsmallmatrix}
                             \theta^i &  & & \\
                              & \theta^{qi} & & \\
                              & & \theta^{-qi} &  \\
                               & & & \theta^{-i} 
                             \end{bsmallmatrix}$ & $q^2-q$ & $\theta^{i(q\pm 1)}\neq 1$  &  $B_2(i)=B_2(qi)$\\

      $\begin{bsmallmatrix}
      1 & & & \\
       & \gamma^i & & \\
       & & \gamma^{-i} & \\
        & & & 1 
      \end{bsmallmatrix}$ & $q^2+q$ & $1\leq i\leq q-2$  & $C_1(i)$ \\
      
      $\begin{bsmallmatrix}
           \gamma^i & & & \\
            & \gamma^{-i} & & \\
            & & \gamma^{i} & \\
             & & & \gamma^{-i}
           \end{bsmallmatrix}$ & $q^2+q$ & $1\leq i\leq q-2$   & $C_2(i)=C_2(-i)$ \\

     $\begin{bsmallmatrix}
           \eta^i & & & \\
            & \eta^{-i} & & \\
            & & \eta^{i} & \\
             & & & \eta^{-i}
           \end{bsmallmatrix}$ & $q^2-q$ & $1\leq i\leq q$  & $C_4(i)= C_4(-i)$ \\

        $\begin{bsmallmatrix}
                      \gamma^i & 1 & & \\
                       & \gamma^i & & \\
                       & & \gamma^{-i} & 1 \\
                        & & & \gamma^{-i}
                      \end{bsmallmatrix} $ & $q^2-1$ & $1\leq i\leq q-2$   &  $D_2(i) $\\
                        
   $ \begin{bsmallmatrix}
                                            1 &  & 1 & \\
                                              & 1 & & 1 \\
                                              & & 1 &  \\
                                               & & & 1
                                             \end{bsmallmatrix}$ & - &   &  $A_{31}$\\
    
      $\begin{bsmallmatrix}
                               1 & 1 & 1 & 1 \\
                                & 1 & &  1\\
                                & & 1 & 1 \\
                                 & & & 1 
                               \end{bsmallmatrix}$ & - &  &  $A_{32}$\\    
    
   $  \begin{bsmallmatrix}
                                                           \gamma^i &  & \gamma^i & \\
                                                            & \gamma^{-i} & & \gamma^{-i} \\
                                                            & & \gamma^{i} &  \\
                                                             & & & \gamma^{-i} 
                                                           \end{bsmallmatrix}$  & - & $1\leq i\leq q-2$  &  $D_2(i) = D_2(-i) $\\

     $  \begin{bsmallmatrix}
                                                               \eta^i &  & \eta^i & \\
                                                                & \eta^{-i} & & \eta^{-i} \\
                                                                & & \eta^{i} &  \\
                                                                 & & & \eta^{-i} 
                                                               \end{bsmallmatrix}$  & - & $1\leq i\leq q$  &  $D_4(i) = D_4(-i) $\\

      \end{tabular}
    \end{center}
  \end{table}
  \clearpage

 From tables \ref{tab:table1} and \ref{tab:table2} and the triviality of the central character of $\pi$ we can conclude the following formulas: when $q$ is even we have \begin{align}\label{q even finite dimension formula}
   (q^2-1)(q^2-q)\dim V^H &=  \chi(A_1)+ (q^2-1) \chi(A_{31}) \nonumber \\
   &+ \frac{q(q+1)}{2}  \sum\limits_{ \substack{i, j=1 \\ i\neq j, q-1-j} }^{q-2}\chi(B_1(i,j))  + \frac{q(q-1)}{2}\sum\limits_{\substack{i=1 \\ (q\pm 1)\nmid i  }}^{q^2-2}\chi(B_2(i)) \nonumber \\
   &+ q(q+1)\sum\limits_{i=1}^{q-2}\chi(C_1(i)) +\frac{q^2+q+2}{2}\sum\limits_{i=1}^{q-2}\chi(C_2(i)) \nonumber \\
   &+ \frac{q(q-1)}{2}\sum\limits_{i=1}^{q}\chi(C_4(i)) + (q^2-1)\sum\limits_{i=1}^{q-2}\chi(D_2(i)). 
   \end{align} and when $q$ is odd we have \begin{align} \label{q odd finite dimension formula}
  (q^2-1)(q^2-q)\dim V ^H & =  \chi(A_0)+ (q^2-1) \chi(A_{21}) +  \frac{q(q+1)}{2} \chi(B_0) + \frac{q(q-1)}{2}  \chi(C_0(\xi^{1/2})) \nonumber \\
  &+ \frac{q^2+q+2}{2}  \sum\limits_{1\neq b\in \Fq^\times} \chi(D_0(1, b)) + (q^2-1)  \sum\limits_{1\neq b\in \Fq^\times} \chi(D_1(1,b)) \nonumber \\
  &+ q(q+1)  \sum\limits_{\pm 1\neq a\in \Fq^\times} \chi(E_0(1, a)) + \frac{q(q-1)}{2} \sum\limits_{a\in \Fq}\chi(F_0(a+\xi^{1/2})) \nonumber \\
  &+ \frac{q(q+1)}{2} \sum\limits_{\substack{  a, b\in \Fq^\times\setminus \{1\} \\  b\neq  a, \ a^2   }}\chi(H_0(1, a, ba^{-1}, b ))  \nonumber \\
 &+ \frac{q(q-1)}{2}  \sum\limits_{\pm 1 \neq b\in \Fq^{\times}, }\chi(I_0(\xi^{1/2}, -b\xi))\nonumber  \\
 &+    \frac{q(q-1)}{2}  \sum\limits_{\substack{ a\in \Fq^\times \\ 1\neq b\in \Fq^\times} }\chi(I_0(a+\xi^{1/2}, b(a^2-\xi)) 
 \end{align} 
 In tables \ref{tab:table3} and \ref{tab:table5} we list the evaluations of the expressions in the right side of (\ref{q even finite dimension formula}) and (\ref{q odd finite dimension formula}) for all of the virtual characters constructed in \cite{En1972} and \cite{Sh1982}. 
 Coupled with table 3 of \cite{Ros2018}, one obtains the $\Si{2}$-columns of tables \ref{tab:table4}, \ref{tab:table6}, and \ref{tab:table7}.


{\bf Proof of signature formulas.} We must determine the distribution of conjugacy classes of $\GSp(4, \Fq)$ in $Mu=uM$. 

 Suppose first that $q$ is even. Elements in $uM$ have the form $\begin{bsmallmatrix}
 & \lambda X\\
 X & 
 \end{bsmallmatrix}$ for $X\in \GL(2, \Fq) $ and $\lambda\in \Fq^\times$. The equalities \begin{eqnarray}
\begin{bsmallmatrix}
   \lambda^{-1/2}I & \\
    &  I
   \end{bsmallmatrix} \begin{bsmallmatrix}
  & \lambda X\\
  X & 
  \end{bsmallmatrix}\begin{bsmallmatrix}
   \lambda^{1/2}I & \\
    &  I
   \end{bsmallmatrix} &=& \lambda^{1/2}\cdot \begin{bsmallmatrix}
     &  X\\
     X & 
     \end{bsmallmatrix} \nonumber \\
     \begin{bsmallmatrix}
       I &  \\
       I & I 
       \end{bsmallmatrix}\begin{bsmallmatrix}
        & X \\
       X &  
       \end{bsmallmatrix}\begin{bsmallmatrix}
       I &  \\
       I & I 
       \end{bsmallmatrix} &=& \begin{bsmallmatrix}
       X & X \\
        & X 
       \end{bsmallmatrix}\nonumber    
 \end{eqnarray} together show that \begin{equation*}
 \sum\limits_{m\in M} \chi(mu) = (q-1)\sum\limits_{X\in \GL(2, \Fq)} \chi( \begin{bsmallmatrix}
 X & X \\ & X
 \end{bsmallmatrix} ) = (q-1)^2\sum\limits_{X\in \SL(2, \Fq)} \chi( \begin{bsmallmatrix}
  X & X \\ & X
  \end{bsmallmatrix} ). 
 \end{equation*} Conjugating by an appropriate  $\begin{bsmallmatrix}
 B & \\
 & B 
 \end{bsmallmatrix}\in M$ allows us to assume that $X$ is in Jordan form. 
From the last four rows in table \ref{tab:table2}, and the orders of centralizers of elements of $\GL(2, \Fq)$, we obtain  \begin{align} \label{finite signature formula q even }
     q(q^2-1) s(\pi) &= \chi(A_{31})+(q^2-1)\chi(A_{32}) +\frac{q(q+1)}{2}\sum\limits_{i=1}^{q-2} \chi(D_2(i))+ \frac{q(q-1)}{2}\sum\limits_{i=1}^q\chi(D_4(i)).
     \end{align} 
   
    Finally, suppose $q$ is odd. Unlike the even $q$ case, not every element of $uM$ is conjugate to one which is block upper triangular. Define the following subgroups and elements of $M$: \begin{eqnarray}
 T = \left\{ \begin{bsmallmatrix}
 1 & & & \\
  & d & & \\
  & & c  & \\
   & & & c d
 \end{bsmallmatrix} \mid c,  d\in \Fq^\times   \right\},  \qquad Z = Z(\Fq), \qquad
 N=   \left\{ \begin{bsmallmatrix}
 1 & b & & \\
  & 1 & & \\
  & & 1 & -b \\
   & & & 1
 \end{bsmallmatrix} \mid b\in \Fq   \right\}, \nonumber 
 \end{eqnarray}
 \begin{eqnarray}
 t_{c,d} = \begin{bsmallmatrix}
  1 & & & \\
   & d & & \\
   & & c  & \\
    & & & c d
  \end{bsmallmatrix},\qquad
 x=  \begin{bsmallmatrix}
 1 & 1 & & \\
  & 1 & & \\
  & & 1 & -1 \\
   & & & 1
 \end{bsmallmatrix} , \qquad 
 w &=& \begin{bsmallmatrix}
 & 1 & & \\
 1 & & & \\
  & & & 1 \\
   & & 1 & 
 \end{bsmallmatrix}. \nonumber
  \end{eqnarray} Clearly $T$ and $u$ normalize $N$ and $\mu(w)=1$. 
  Using the Bruhat decomposition, every element in $M$ may be written uniquely in the form $ztn$ or in the form $ztnwn'$ for $z\in Z$, $t\in T$, $n,n'\in N$. Since $ztnwn'u = ztnwu(u^{-1}n'u)$ is conjugate to $zt(t^{-1}u^{-1}n'u t n)wu$, we obtain 
  \begin{align}
 \sum\limits_{m\in M} \chi(mu) &= (q-1)\sum\limits_{t\in T, n\in N} \left( \chi(tnu) + q\chi(tnwu)  \right) 
  \end{align} Note that $\mu(t_{c,d}nu)=\mu(t_{c,d}nwu)=cd$. The reader can readily verify the following assertions:  \begin{enumerate}
  \item We have $t_{a^2 c, d}nu$ is $T$-conjugate to $aI_4 \cdot t_{c,d}nu $ and $t_{a^2 c, d}nwu$ is $T$-conjugate to $aI_4 \cdot t_{c,d}nwu $, 
  
  \item If $1\neq n\in N$ has corresponding parameter $b\neq 0$ then $t_{c,d}nu$ is $T$-conjugate to $t_{c,d}xu$ and $t_{c, b^2 d}nwu$ is $T$-conjugate to $bI_4 \cdot t_{c,d}xwu$, 
  
  \item If $d\neq 1$ then $t_{c,d}nu$ is $N$-conjugate to $t_{c,d}u$. 
  \end{enumerate}
  
  Thus 
  \begin{align}
  \sum\limits_{m\in M} \chi(mu) &= \frac{(q-1)^2}{2}\sum\limits_{c\in \{ 1, \xi \}, d\in \Fq^\times, n\in N} \left( \chi(t_{c,d}nu) + q\chi(t_{c,d}nwu)  \right)\nonumber\\
  &= \frac{(q-1)^2}{2}\sum\limits_{c\in \{ 1, \xi \}, d\in \Fq^\times} \left( \chi(t_{c,d}u)+(q-1)\chi(t_{c,d}xu) + q\sum\limits_{n\in N}\chi(t_{c,d}nwu)  \right)\nonumber\\
  &= \frac{(q-1)^2}{2}\sum\limits_{c\in \{ 1, \xi \} }  \Bigg( \chi(t_{c,1}u) +(q-1)\chi(t_{c,1}xu)+q\sum\limits_{1\neq d\in \Fq^\times}  \chi(t_{c,d}u)  \nonumber\\
  & \hspace{.5 cm} + q\sum\limits_{d\in \Fq^\times }\left(\chi(t_{c,d}wu) + 
  (q-1)\chi(t_{c,d}xwu) \right)  \Bigg). 
  \end{align}  We now record the following facts, referring to \cite{Sh1982} for the notation for conjugacy classes: \begin{equation}
       t_{c,d}u \text{ is in a conjugacy class of type }
       \begin{cases}
        B_0(1) &\text{if } c=d=1\\
        C_0(\xi^{1/2}) &\text{if } c=\xi, d=1\\
        D_0(1,-1) &\text{if }  c=1, d=-1 \\
         F_0(\xi^{1/2}) &\text{if } c=\xi, d=-1\\
         H_0(1,-1, -d, d) & \text{if } c=1,  d\neq \pm 1\\
             I_0(\xi^{1/2},  d\xi) &\text{if } c=\xi, d\neq \pm 1\\
            \end{cases},
           \end{equation}

             \begin{equation}
                       t_{c,1}xu \text{ is in a conjugacy class of type }
                       \begin{cases}
                       B_{31}(1) &\text{if } c=1, -1\in \Fq^{\times 2}\\
                        B_{32}(1) &\text{if } c=1, -1\not\in \Fq^{\times 2}\\
                                    C_{11}(\xi^{1/2}) &\text{if } c=\xi, -1\not\in \Fq^{\times 2}\\
                                     C_{12}(\xi^{1/2}) &\text{if } c=\xi, -1\in \Fq^{\times 2}\\
                       \end{cases},
                      \end{equation}

               \begin{equation}
                                    t_{c,d}wu \text{ is in a conjugacy class of type }
                                    \begin{cases}
                                    D_0((-cd)^{1/2},-(-cd)^{1/2}) & \text{if } -cd\in \Fq^{\times 2}\\
                                                               F_0((-cd)^{1/2})
                                                                        & \text{if } -cd\not\in \Fq^{\times 2} \\
                                    \end{cases}.
                                   \end{equation}            For the last elements, observe that $t_{c, d}xwu$ has eigenvalues $(-c)^{1/2}\left(\frac{\pm 1 \pm (1+4d)^{1/2} }{2}\right)$. 
 In particular there are four distinct eigenvalues unless $d=-1/4$. It is not hard to show that 
 
  \begin{equation}
               t_{c,-1/4}xwu \text{ is in a conjugacy class of type }
                \begin{cases}
                 B_{31}(1/2) &\text{if }  -c, -1\in \Fq^{\times 2},  \\
                  B_{32}(1/2) &\text{if }  -c\in \Fq^{\times 2}, -1\not\in \Fq^{\times 2} \\
                   C_{11}(\xi^{1/2}/2)    &\text{if }  -c, -1\not\in \Fq^{\times 2}. \\
                     C_{12}(\xi^{1/2}/2)  &\text{if }  -c\not\in \Fq^{\times 2}, -1\in \Fq^{\times 2}. \\ 
             \end{cases}               \end{equation}
  Note that $\{-1, -\xi\}$ is a set of representatives for $\Fq^\times/\Fq^{\times 2}$. Finally we consider $t_{c,d}xwu$ with $d\neq -1/4$. For $r\in \mathbb{F}_{q^2}^\times$ we have $r^2\in \Fq$ if and only if $r^{q-1}\in \{ \pm1 \}$, with the sign determined by whether $r\in \Fq^\times$ or not. 
   Thus \begin{align}
   \sum\limits_{d\in \Fq^\times, d\neq -1/4} \chi(t_{c,d}xwu) &=\sum\limits_{d\in \Fq^\times, d\neq -1} \chi(t_{c,d/4}xwu) \\
    &=\frac{1}{2} \sum\limits_{  \substack{ r\in \mathbb{F}_{q^2}^{\times} \setminus\{\pm1  \}  \\ r^{q-1}=\pm 1   }     } \chi(t_{c, \frac{r^2-1}{4}} nwu).
   \end{align} The eigenvalues of $t_{c, \frac{r^2-1}{4}} nwu$ are $(-c)^{1/2}\left(\frac{\pm 1\pm r}{2}\right)$. It is not hard to show that the conjugacy class of $t_{c,\frac{r^2-1}{4}  }xwu$ is of type               
                                                          
                                                       \begin{equation}
                                                       \begin{cases}
                                                      
                                                 (-c)^{1/2}\cdot  H_0(\frac{1+r}{2}, \frac{-1-r}{2}, \frac{-1+r}{2} , \frac{1-r}{2}  )   &\text{if } -c\in \Fq^{\times 2}, r\in \Fq^\times\\  
                                                 (-c)^{1/2} L_0(\frac{1+r}{2}, \frac{-1+r}{2})   &\text{if }  -c\in \Fq^{\times 2}, r\not\in \Fq^\times \\
                     I_0((-c)^{1/2}\left( \frac{1+r}{2}  \right), \frac{c(r^2-1)}{4})   &\text{if }  -c\not\in \Fq^{\times 2},  r\in \Fq^\times \\     
                                                    I_0((-c)^{1/2}\left( \frac{1+r}{2}  \right), \frac{c(r^2-1)}{4})   &\text{if }  -c\not\in \Fq^{\times 2},  r\not\in \Fq^\times 
                                               \end{cases}.
                                                      \end{equation}                      
 Here the conjugacy class $I_0(a,b)$ denotes the one with an eigenvalue $a\not\in \Fq$ and similitude $b\neq a^2, a^{q+1}$, and $L_0(a,b)$ denotes a conjugacy class with eigenvalues $a,b\not\in \Fq$ and $a\neq b^q$. Putting everything together and using the triviality of the central character, we obtain

%
 
 
  \begin{align}\label{finite signature formula q odd }
  2q(q^2-1)s(\pi)&= \chi(B_0) +  \chi(C_0(\xi^{1/2})) \nonumber \\
  &+  (q^2-1) \left( \chi(B_{31}) \left(\frac{1+\alpha_0(-1)}{2} \right)+\chi(B_{32}) \left(\frac{1-\alpha_0(-1)}{2} \right) \right)\nonumber \\
  &+ (q^2-1)\left( \chi(C_{11}(\xi^{1/2}))\left( \frac{1-\alpha_0(-1)}{2}\right)+\chi(C_{12}(\xi^{1/2}))\left( \frac{1+\alpha_0(-1)}{2} \right)  \right)\nonumber \\
  &+ q^2\left(\chi(D_0(1,-1))+\chi(F_0(\xi^{1/2}))\right)  \nonumber \\
  &+  \frac{q(q+1)}{2} \sum\limits_{\pm 1 \neq a\in \Fq^\times} \left( \chi(H_0(1, -1, -a, a))+ \chi (I_0(\xi^{1/2}, \xi a)) \right)  \nonumber \\
  &+ \frac{q(q-1)}{2}\sum\limits_{a\in \Fq^\times }\left(  \chi(I_0(1+a\xi^{1/2}, -1+a^2\xi)) +\chi(L_0(1+a\xi^{1/2}, -1-a\xi^{1/2})) \right). 
  \end{align}  
  

In tables \ref{tab:table3} and \ref{tab:table5} we list the evaluations of the expressions in the right side of (\ref{finite signature formula q even }) and (\ref{finite signature formula q odd }) for all of the virtual characters given in the tables of \cite{En1972} and \cite{Sh1982}. 
Coupled with table 3 of \cite{Ros2018}, one obtains the $s$-columns of tables \ref{tab:table4}, \ref{tab:table6}, and \ref{tab:table7}. 
\end{proof}

\begin{corollary} Let $(\pi, V)$ be an irreducible representation of $G$ with trivial central character 
 Then $\dim V^{\Si{2}} \in \{0, 1, 2, 3, 4, 5, 6, 7, 8, 12  \}. $
\end{corollary}



\section{Tables}




The irreducible characters of $\GSp(4, q)$, with $q$ even, are classified in \cite{En1972}, and those with $q$ odd are handled in \cite{Sh1982}. We refer to their papers for the definition of the virtual characters that appear in the first column of tables \ref{tab:table3} and \ref{tab:table5}. 
In table \ref{tab:table3} we have listed the parameters for each character. In table \ref{tab:table5}, when $q$ is odd, the parameters are as follows: $\lambda, \mu,\nu: \Fq^\times \to \C^\times$, $\Lambda:\mathbb{F}_{q^2}^\times \to \C^\times$, $\omega : \mathbb{F}_{q^2}^\times[q+1]\to \C^\times $, $\Theta: \mathbb{F}_{q^2}^\times[(q-1)(q^2+1)]\to \C^\times$, and $\lambda': \mathbb{F}_{q^2}^\times[2(q-1)]\to \C^\times$. The character $\omega_0:\mathbb{F}_{q^2}^\times[q+1]\to\{\pm 1\} $ is the nontrivial quadratic character, and $N: \mathbb{F}_{q^2}^\times\to \Fq^\times$ denotes the norm map. The $\delta_{-,-}$ denotes the Kronecker delta. 

To allow for reducible representations of $\GSp(4, \Fq)$, we have not restricted the parameters defining these characters. For example, the unramified parabolically induced representations $V$ of $\GSp(4, F)$ will have reducible parahoric restriction $V^{K^+}$ as a $K$-representation. Some choices of parameters can result in non-genuine representations, hence the possibility of negative entries in the $M$ columns of tables \ref{tab:table3} and \ref{tab:table5}.

The irreducible nonsupercuspidal representations of $\GSp(4, F)$ were classified in \cite{Sa1993}. The families they construct constitute the representations given in families I, II,$\ldots,$ XI in tables \ref{tab:table4}, \ref{tab:table6} and \ref{tab:table7}. There are also three families of supercuspidals that appear in these tables in the last rows. The characters $\chi_i, \chi, \sigma:F^\times\to \C^\times$ and quadratic characters $\xi:F^\times \to \{\pm 1\}$ are all tamely ramified. We write $a(\sigma)\in \{ 0,1  \}$ for the conductor of a tamely ramified character $\sigma$. The letter $\rho$ indicates a depth zero supercuspidal representation of $\GL(2, F)$, to which one may associate a parameter $\Lambda$ as above with $\Lambda^{q-1}\neq 1$, and $\omega_\Lambda:=\Lambda^{q-1}$. 

In table \ref{tab:table8} and \ref{tab:table9}, $\zeta_n$ denotes a complex primitive $n$th root of unity.

 \tabcolsep=0.11cm
\begin{table}[h!]
  \begin{center}
    \caption{The ``dimensions'' of $M$-fixed vectors and signature $s$ of the involution $u$ for  $q$ even.}
    \label{tab:table3}
    \begin{tabular}{c|c|c|c} 
   $\sigma$ & 
      parameters & $M$  &     $s$ \\
      \hline
      
    $\theta_0$ &    & 1 &     1   \\
         
            $\theta_1$ &    & 3   &     1     \\
     
  $\theta_2$ &    & 3  &   1     \\
       
   $\theta_3$ &    & 0 &  0       \\
   
   $\theta_4$ &    &  2 &  0     \\
   
   $\theta_5$ &    & 0 &   0      \\
 
 $\chi_1(k,l)$ &  $k,l\in \Z/(q-1)\Z$  & $2+2(\delta_{k,0}+\delta_{l,0})^2+\delta_{k,l}+\delta_{k,-l}$  &    $2+\delta_{k,l}+\delta_{k,-l}$    \\

    $\chi_2(k)$ & $k\in \Z/(q^2-1)\Z$   & $\delta_{(q+1)k,0}-\delta_{(q-1)k,0}-2\delta_{k,0}$ &  $-\delta_{(q+1)k,0}-\delta_{(q-1)k,0}$      \\

     $\chi_3(k,l)$ &  $\substack{ k\in \Z/(q-1)\Z \\ l\in \Z/(q+1)/\Z}$     & $2+2\delta_{k,0}$  & $0$     \\

      $\chi_4(k,l)$ & $k, l\in \Z/(q+1)/\Z$    & $2-\delta_{k,l}-\delta_{k,-l}$  &  $-2+\delta_{k,l}+\delta_{k,-l}$     \\

  $\chi_5(k)$ & $k\in \Z/(q^2+1)/\Z$ &  0 &   0    \\

    $\chi_6(k)$ & $k\in \Z/(q-1)\Z$  &  $2+5\delta_{k,0}$  &       $2+\delta_{k,0}$  \\

   $\chi_7(k)$ & $k\in \Z/(q-1)\Z$  & $1+3\delta_{k,0}$ &   $1+\delta_{k,0}$     \\

   $\chi_8(k)$ & $k\in \Z/(q+1)\Z$  & $-\delta_{k,0}$   &   $-\delta_{k,0}$      \\

  $\chi_9(k)$ & $k\in \Z/(q+1)/\Z$  & $1+\delta_{k,0}$  &  $1-\delta_{k,0}$     \\

           $\chi_{10}(k)$ & $k\in \Z/(q-1)\Z$  &   $1+4\delta_{k,0}$ &    $1$   \\

          $\chi_{11}(k)$ & $k\in \Z/(q-1)\Z$  &  $3+5\delta_{k,0}$    &  $1+\delta_{k,0}$    \\

   $\chi_{12}(k)$ & $k\in \Z/(q+1)\Z$  &  $1-2\delta_{k,0}$ &  $-1$   \\

  $\chi_{13}(k)$ & $k\in \Z/(q+1)\Z$  & $3-\delta_{k,0}$  &   $-1+\delta_{k,0}$   \\

    \end{tabular}
  \end{center}
\end{table}


 \tabcolsep=0.11cm
\begin{table}[h!]
  \begin{center}
    \caption{$\Si{2}$-fixed vector dimensions and Atkin-Lehner signatures, $q$ even. }
    \label{tab:table4}
    \begin{tabular}{c|c|c|c|c} 
Type &   $\pi$  &
      data &  $\Si{2}$   &     $s$\\
      \hline
      
I &       $\chi_1\times\chi_2\rtimes \sigma$  & $a(\chi_1)=a(\chi_2)=0$ &   12 &      4  \\

&     &   $a(\chi_1)+a(\chi_2)=1$ &   4 &      $2$  \\
       
 &    &   $a(\chi_i)=a(\sigma)+a(\chi_i\sigma)=1$ & $3$ &      $3$  \\

 &     &   $a(\chi_i)=a(\sigma)=a(\chi_i\sigma)=1$ & $2$ &      $2$  \\

   IIa &$\chi \St_{\GL(2)}\rtimes \sigma$  & $a(\chi)=0$  & 5  &  1   \\  
   
 &    & $a(\chi)=1$ & 1  &   1  \\       
      
   IIb& $\chi 1_{\GL(2)}\rtimes \sigma$ & $a(\chi)=0$ &  7&  3    \\   
   
  &   &  $a(\chi)=1$ & 2  & 2   \\        
       
   IIIa &$\chi\rtimes \sigma \St_{\GL(2)}$ & $a(\chi)=0$ & 8 &  2    \\    
   
 &  & $a(\chi)=1$ & 3 &  1    \\       
            
   IIIb&  $\chi\rtimes \sigma 1_{\GL(2)}$  &  $a(\chi)=0$ &  4&  2    \\ 
   
 &   & $a(\chi)=1$ & 1  &  1   \\

    IVa& $\sigma \St_G$&$a(\sigma)=0$ &  2 &     0  \\

      IVb &$L(\nu^{2}, \nu^{-1}\sigma \St_{\GSp(2)})$  & $a(\sigma)=0$ & 6  &   2  \\

      IVc &$L(\nu^{3/2} \St_{\GL(2)}, \nu^{-3/2}\sigma)$ &  $a(\sigma)=0$ & 3  &  1   \\

      IVd &$\sigma 1_G$  & $a(\sigma)=0$   &   1 &   1   \\

         Va& $\delta( [\xi, \nu\xi], \nu^{-1/2}\sigma  )$ & $a(\sigma)=0$  & 2 &  0   \\        
               
            Vb &$L(\nu^{1/2}\xi\St_{\GL(2)}, \nu^{-1/2}\sigma)$ & $a(\sigma)=0$& 3  &    1\\

            Vc& $L(\nu^{1/2}\xi\St_{\GL(2)}, \nu^{-1/2}\xi\sigma)$  & $a(\sigma)=0$ & 3  &   1  \\

            Vd & $L(\nu\xi, \xi\rtimes \nu^{-1/2}\sigma)$ & $a(\sigma)=0$ & 4  &  2    \\

   VIa &$\tau(S,\nu^{-1/2}\sigma)$ & $a(\sigma)=0$ &  $5$  & 1   \\

 VIb &$\tau(T,\nu^{-1/2}\sigma)$&  $a(\sigma)=0$ & 3 &  1   \\

 VIc& $L(\nu^{1/2}\St_{\GL(2)},\nu^{-1/2}\sigma )$ &  $a(\sigma)=0$ &  0 & 0   \\

   VId& $L(\nu, 1_{F^\times}\rtimes \nu^{-1/2}\sigma)$  & $a(\sigma)=0$ & $4$ &    2  \\

   VII &$\chi\rtimes \rho$ &  $a(\chi)=0$ & 4  &    0 \\   
   
 &  &    $a(\chi)=1$ & 2  &  0   \\        
               
   VIIIa&  $\tau(S, \rho)$  &  & 3 &  $-1$    \\        
               
  VIIIb& $\tau(T, \rho)$  & & 1 &   1  \\        
                     
   IXa& $\delta(\nu\xi, \nu^{-1/2}\rho)$  & & 3 &  $-1$   \\  
   
    IXb&  $L(\nu\xi, \nu^{-1/2}\rho)$& & 1 &  1    \\

   X& $\rho\rtimes \sigma$ &  $a(\sigma)=0$ & $1$  &    $-1$  \\  
   
&   &   $a(\sigma)=1$ & $0$  & $0$    \\

   XIa& $\delta(\nu^{1/2}\rho, \nu^{-1/2}\sigma)$ & $a(\sigma)=0$ &  1 & $-1$    \\  
      
       XIb& $L(\nu^{1/2}\rho, \nu^{-1/2}\sigma)$ & $a(\sigma)=0$ &  0 &   0    \\  
      
     s.c. &  $\cInd_{ZK}^G \chi_5(k)$ &  & $0$ &  0   \\
     
       s.c.& $\cInd_{ZK}^G\chi_4(k,l)$ &  &$2$ &  $-2$ \\
     
       s.c. & $\cInd_{ZK}^G\theta_5$&   & 0 &  0  
      
    \end{tabular}
  \end{center}
\end{table}

\begin{landscape}

 \tabcolsep=0.11cm
\begin{table}[h!]
  \begin{center}
    \caption{The ``dimensions'' of $M$-fixed vectors and signature $s$ of the involution $u$, for $q$ odd. All representations are assumed to have trivial central character. 
    }
    \label{tab:table5}
    \begin{tabular}{c|c|c|c} 
   $\sigma$ & $\omega_\sigma$ &  $M$  &     $s$ \\
      \hline

 $X_1(\lambda,\mu, \nu)$ &  $\lambda\mu\nu^2$  & $1+\lambda(-1)+\delta_{\nu,1}+\delta_{\lambda\nu,1}+2(\delta_{\lambda,1}+\delta_{\mu,1})+4\delta_{\lambda,1}\delta_{\nu,1}$  &  $\nu(-1)(1+\delta_{\lambda\nu,1})+\lambda\nu(-1)(1+\delta_{\nu,1})$      \\

 $X_2(\Lambda, \nu)$ & $\nu^2\cdot \Lambda|_{\Fq^\times}$ &  $1-\Lambda(\xi^{1/2})\nu(\xi)+\delta_{\nu,1}-\delta_{\Lambda \nu\circ N, 1}-2\delta_{\nu,1}\delta_{\Lambda,1}$ &  $\nu(-1)\left(1 - \Lambda(\xi^{1/2})\nu(\xi)(1+\delta_{\nu,1})-\delta_{\Lambda \nu\circ N, 1}\right)$         \\

 $X_3(\Lambda, \nu)$  & $\nu\cdot \Lambda|_{\Fq^\times}$ &   $1+\nu(-1)+2\delta_{\nu,1}$  &  $0$       \\

    $X_4(\Theta)$ & $\Theta|_{\Fq^\times}$  & $1-\Theta(\xi^{1/2})$   & 0     \\  
    
      $X_5(\Lambda, \omega)$ &  $\Lambda|_{\Fq^\times}$  & $1+\omega(-1)-\delta_{\Lambda,1}-\delta_{\Lambda\widetilde{\omega},1}$  & $-\Lambda(\xi^{1/2})(1+\omega(-1))+\omega(-1)(\delta_{\Lambda,1}+\delta_{\Lambda \widetilde{\omega},1 })$    \\

      $\chi_1(\lambda, \nu)$ & $\lambda\nu^2$ & $1+3\delta_{\nu,1} $ & $\nu(-1)+\delta_{\nu,1}$     \\

      $\chi_2(\lambda, \nu)$ &  $\lambda\nu^2$ & $3+3\delta_{\nu,1}+2\delta_{\nu^2,1}$ &   $\nu(-1)+\delta_{\nu,1}$      \\

             $\chi_3(\lambda, \nu)$ &$\lambda^2\nu^2$ & $\delta_{\lambda\nu, 1}(1+\lambda(-1))+2\delta_{\lambda,1}+3\delta_{\lambda,1}\delta_{\nu,1}$  & $\delta_{\lambda\nu,1}(1+\lambda(-1))+\delta_{\lambda,1}\delta_{\nu,1}$        \\

    $\chi_4(\lambda, \nu)$ &$\lambda^2\nu^2$ & $1+\lambda(-1)\delta_{\lambda\nu,\alpha_0}+\delta_{\nu,1}+2\delta_{\lambda,1}+\delta_{\lambda,1}\delta_{\nu,1}$  &   $\nu(-1)+\alpha_0(-1) (1+\delta_{\nu,1})\delta_{\lambda\nu,\alpha_0}$      \\

      $\chi_5(\omega, \nu)$ &$\nu^2$ & $\delta_{\nu,1}(1-\omega(-1)-\delta_{\omega,1})$  &     $\delta_{\nu,1}(1-\omega(-1)-\delta_{\omega,1})$    \\

       $\chi_6(\omega, \nu)$ &$\nu^2$ & $1+\delta_{\nu, \alpha_0}(\omega(-1)-\delta_{\omega,\omega_0})-2\delta_{\nu,1}\delta_{\omega,1}$ &  $-\omega(-1)\nu(-\xi)+\nu(-1)\delta_{\nu,\alpha_0}(1-\delta_{\omega,\omega_0})$        \\   
       
          $\chi_7(\Lambda)$ & $\Lambda|_{\Fq^\times}$ &  $1+\delta_{\Lambda,1}$   &   $\Lambda(\xi^{1/2})-\delta_{\Lambda,1}$      \\

               $\chi_8(\Lambda)$ &$\Lambda|_{\Fq^\times}$ &  $3-\delta_{\Lambda,1}$  & $-\Lambda(\xi^{1/2})+\delta_{\Lambda,1}$         \\

      $\tau_1$ & 1 & $0$ & $0$        \\

         $\tau_2(\nu)$ & $\nu^2$ & $\delta_{ \nu,1}(1+\alpha_0(-1))$  &   $\delta_{ \nu,1}(1+\alpha_0(-1))$        \\

            $\tau_3$ & 1 & $1$  &   $\alpha_0(-1)$       \\

                $\tau_4(\lambda')$ & $\lambda'|_{\Fq^\times}$ &  $0$  & $0$        \\

                       $\tau_5(\lambda')$ & $\lambda'|_{\Fq^\times}$& $0$ & $0$        \\

         $\theta_1(\nu)$ &$\nu^2$ &  $1+2\delta_{\nu,1}$  & $\nu(-1)$        \\ 
         
           $\theta_2(\nu)$ & $\nu^2$ & $0$  &  $0$        \\       
      
        $\theta_3(\nu)$ &$\nu^2$ &  $1+2\delta_{ \nu,1}$  & $-\nu(-1)+2\delta_{\nu,1}$         \\ 
        
  $\theta_4(\nu)$ & $\nu^2$ &  $0$ &  $0$        \\ 
     
       $\theta_5(\nu)$ &$\nu^2$ &  $3-\delta_{ \nu,1}$  & $\nu(-1)-\delta_{\nu,1}$         \\  
       
         $\theta_0(\nu)$ & $\nu^2$ & $\delta_{\nu,1}$  & $\delta_{ \nu,1}$        \\ 
     
    \end{tabular}
  \end{center}
\end{table}

\end{landscape}

\begin{landscape}
 \tabcolsep=0.11cm
\begin{table}[h!]
  \begin{center}
    \caption{$\Si{2}$-fixed vector dimensions and Atkin-Lehner signatures, $q$ odd.}
    \label{tab:table6}
    \begin{tabular}{c|c|c|c|c|c} 
 Type &   $\pi$ &  $\omega_\pi$ &
      conditions & $\Si{2}$   &    $s$\\
      \hline
      
I &        $\chi_1\times\chi_2\rtimes \sigma$ & $\chi_1\chi_2\sigma^2$ & $a(\sigma)=a(\chi_i)=0$ & 12 &      4  \\
   
   &   &  & $a(\sigma)=1$, $a(\chi_i)=0$ & 6 &      $2\sigma(-1)$  \\
     
    &  &   & $a(\sigma) = a(\chi_1)+a(\chi_2)=1$ & 4 &      $2\sigma(-1)$  \\
       
   &   &   &$a(\chi_i)=a(\sigma)+a(\chi_i\sigma)=1$ & $2+\chi_1(-1)$ &      $1+2\chi_1(-1)$  \\

 &      &  &  $a(\chi_i)=a(\sigma)=a(\chi_i\sigma)=1$ & $1+\chi_1(-1)$ &      $(1+\chi_1(-1))\sigma(-1)$  \\

   IIa&  $\chi \St_{\GL(2)}\rtimes \sigma$  & $\chi^2\sigma^2$ & $a(\sigma)=a(\chi)=0$ & 5 &  1   \\      
   
 &   & & $a(\sigma)=0$, $a(\chi)=1$ & $2+\chi(-1)$  &   $1+2\chi(-1)$     \\ 
    
 &    &  & $a(\sigma)=1$, $a(\chi)=0$ & 4 &   $2\sigma(-1)$    \\

   &    &   & $a(\sigma)=a(\chi)=1$, $a(\sigma\chi)=0$ & 1 &    $\sigma(-1)$  \\   
           
    &          &   & $a(\sigma)=a(\chi)=a(\sigma\chi)=1$ & $1+\chi(-1)$ &   $(1+\chi(-1))\sigma(-1)$   \\        
      
   IIb&  $\chi 1_{\GL(2)}\rtimes \sigma$ &  $\chi^2\sigma^2$ & $a(\sigma)=a(\chi)=0$ & 7 &   3  \\  
   
& &   & $a(\sigma)=0$, $a(\chi)=1$ & 0 &  0     \\ 
       
 &      &   & $a(\sigma)=1$, $a(\chi)=0$ & 2 &  0      \\

  &       &   & $a(\sigma)=a(\chi)=1$, $a(\sigma\chi)=0$ & $1+\chi(-1)$ &      $1+\chi(-1)$ \\   
        
    &       &   & $a(\sigma)=a(\chi)=a(\sigma\chi)=1$ & 0 &      0\\       
       
   IIIa&  $\chi\rtimes \sigma \St$ & $\chi\sigma^2$ & $a(\sigma)=a(\chi)=0$ & $8$ &  2   \\     
   
  &   &  & $a(\sigma)=1$, $a(\chi)=0$ & $5$   &   $\sigma(-1)$   \\     
    
 &    &  & $a(\sigma)=a(\chi)=1$ & $3$  &   $\sigma(-1)$  \\           
            
   IIIb & $\chi\rtimes \sigma 1 $ & $\chi\sigma^2$&  $a(\sigma)=0$ & 4  &  2  \\

   &    & & $a(\sigma)=1$ & 1  &  $\sigma(-1)$    \\


    IVa&  $\sigma \St_G$& $\sigma^2$ & $a(\sigma)=0$ &  2 &     0  \\        
    
   &   &  & $a(\sigma)=1$ &  3 &  $\sigma(-1)$   \\        
         
      IVb & $L(\nu^{2}, \nu^{-1}\sigma \St_{\GSp(2)})$ & $\sigma^2$ & $a(\sigma)=0$ & 6  &    2  \\    
      
    &   & & $a(\sigma)=1$ & 2  &   0  \\        
          
      IVc & $L(\nu^{3/2} \St_{\GL(2)}, \nu^{-3/2}\sigma)$ & $\sigma^2$ & $a(\sigma)=0$ & 3  &  1   \\

     &  &  & $a(\sigma)=1$ & 1 &  $\sigma(-1)$  \\        
              
      IVd&  $\sigma 1_G$ & $\sigma^2$ & $a(\sigma)=0$   &   1 & 1   \\ 
      
     &  &  & $a(\sigma)=1$ & 0 &  0

    \end{tabular}
  \end{center}
\end{table}

\end{landscape}

 \tabcolsep=0.11cm
\begin{table}[h!]
  \begin{center}
    \caption{$\Si{2}$-fixed vector dimensions and Atkin-Lehner signatures, $q$ odd, continued. Here $t\in \mathbb{F}_{q^2}^\times$ denotes any element with $t^{q-1}=-1$. }
    \label{tab:table7}
    \begin{tabular}{c|c|c|c|c|c} 
 Type  &  $\pi$ &  $\omega_\pi$ &
      conditions & $\Si{2}$   &     $s$\\
      \hline

            Va & $\delta([\xi, \nu\xi], \nu^{-1/2}\sigma)$ & $\sigma^2$ &  $a(\xi)=a(\sigma)=0$ & 2 &     0 \\        
            
           &      & &  $a(\xi)=0$, $a(\sigma)=1$ & 3 &  $\sigma(-1)$    \\  
                
          &     & & $a(\xi)=1$  & 1 &   $\xi(-1)$  \\        
                         
                  
               Vb & $L(\nu^{1/2}\xi\St_{\GL(2)}, \nu^{-1/2}\sigma)$ & $\sigma^2$ & $a(\xi)=a(\sigma)=0$& 3  &   1 \\      
               
             &       &  & $a(\xi)=0$, $a(\sigma)=1$& 1  &   $\sigma(-1)$  \\

          &       &  &  $a(\xi)=1$, $a(\sigma)=0$& $1+\xi(-1)$ &  $1+\xi(-1)$   \\  
               
          &       &  &  $a(\xi)=a(\sigma)=1$& 0  &   0  \\             
                   
               Vc & $L(\nu^{1/2}\xi\St_{\GL(2)}, \nu^{-1/2}\xi\sigma)$  & $\sigma^2$ & $a(\xi)=a(\sigma)=0$ & 3  &   1  \\

              &     &  & $a(\xi)=0$, $a(\sigma)=1$ & 1   &   $\sigma(-1)$   \\

           &    & & $a(\xi)=1$ $a(\sigma)=0$ & 0 &     0 \\   
               
          &        &  & $a(\xi)=a(\sigma)=1$& $1+\sigma(-1)$&   $1+\sigma(-1)$  \\        
                        
               Vd&  $L(\nu\xi, \xi\rtimes \nu^{-1/2}\sigma)$ & $\sigma^2$ & $a(\xi)=a(\sigma)=0$& $4$  &  2    \\  
               
                  &   & & $a(\xi)=0$, $a(\sigma)=1$ & $1$  & $\sigma(-1)$    \\  
                    
                &      &  & $a(\xi)=1$ & 0  &   0  \\  
                                

      VIa & $\tau(S,\nu^{-1/2}\sigma)$ & $\sigma^2$ & $a(\sigma)=0$ &  $5$  &  1    \\    
       
      &  & &$a(\sigma)=1$ & $4$  &  $2\sigma(-1)$    \\        
                   
     VIb & $\tau(T,\nu^{-1/2}\sigma)$& $\sigma^2$ & $a(\sigma)=0$ & 3 &  1  \\        
     
    &  & & $a(\sigma)=1$ & 1 &   $-\sigma(-1)$  \\    
                    
     VIc&  $L(\nu^{1/2}\St_{\GL(2)},\nu^{-1/2}\sigma )$ & $\sigma^2$  &&  0 &  0   \\     
     
                         
       VId&  $L(\nu, 1_{F^\times}\rtimes \nu^{-1/2}\sigma)$ & & $a(\sigma)=0$ & $4$ &    2  \\  
       
      &  & & $a(\sigma)=1$ & $1$ & $\sigma(-1)$   \\

     VII&  $\chi\rtimes \rho$ & $\chi\omega_\rho$  & $a(\chi)=0$ & 4  &    0 \\   
       
    &    &   & $a(\chi)=1$ & $1+\chi(-1)$  & 0   \\        
                   
       VIIIa &  $\tau(S,\rho)$ & $\omega_\rho$ & & 3  &  $-\Lambda(t)$    \\        
                    
      VIIIb & $\tau(T, \rho)$& $\omega_\rho$ & & 1 &   $\Lambda(t)$  \\        
                         
       IXa & $\delta(\nu\xi, \nu^{-1/2}\rho)$ & $\xi \omega_\rho$ & $a(\xi)=0$ & 3   &  $-\Lambda(t)$   \\

         &   
         & & $a(\xi)=1$ & 0  &  0  \\

        IXb &   $L(\nu\xi, \nu^{-1/2}\rho)$ & $\xi \omega_\rho$ & $a(\xi)=0$ & 1 &  $-\Lambda(t)$    \\  
        
        &  
        & & $a(\xi)=1$ &  0 &  0   \\

   X&  $\rho\rtimes \sigma$ & $\omega_\rho\sigma^2$ & $a(\sigma)=0$ & $2-\Lambda(t)$  &    $1-2\Lambda(t)$  \\  
   
  &  &  & $a(\sigma)=1$ & $1-\Lambda\sigma^2(t)$  &  $\sigma(-1)(1-\Lambda\sigma^2(t))$    \\

   XIa &  $\delta(\nu^{1/2}\rho, \nu^{-1/2}\sigma)$  & $\sigma^2$ & $a(\sigma)=0$ & 1   &    $-\omega_\Lambda(-1)$  \\  
   
 &     & $\sigma^2$ & $a(\sigma)=1$ & $1+\omega_\Lambda(-1)$  &   $\sigma(-1)(1+\omega_\Lambda(-1))$    \\ 
      
       XIb&   $L(.\nu^{1/2}\rho, \nu^{-1/2}\sigma)$ & $\sigma^2$ &$a(\sigma)=0$  & $1-\omega_\Lambda(-1)$  &  $1-\omega_\Lambda(-1)$  \\

     &    & $\sigma^2$ &$a(\sigma)=1$  & 0 &    0 \\  
      
     s.c. &  $\cInd_{ZK}^GX_4(\Theta)$  & $\widetilde{\Theta}$ & &$1-\Theta(t)$ & 0   \\
     
       s.c. &  $\cInd_{ZK}^G X_5(\Lambda,\omega)$& $\widetilde{\Lambda}$ &  &$1+\omega(-1)$ &   $-\Lambda(t)(1+\omega(-1))$ \\
     
       s.c. &  $\cInd_{ZK}^G\theta_2$& $1$ &  & 0 &  0  
      
    \end{tabular}
  \end{center}
\end{table}

 \tabcolsep=0.11cm
\begin{table}[h]
  \begin{center}
    \caption{character table for $\Sp(4, q)$, $q$ even; $\alpha_i = \zeta_{q-1}^i + \zeta_{q-1}^{-i}$, $\beta_i = \zeta_{q+1}^i + \zeta_{q+1}^{-i}$,  $\theta_i = \zeta_{q^2-1}^i + \zeta_{q^2-1}^{-i}$}
    \label{tab:table8}
    \begin{adjustbox}{max width=\textwidth,max totalheight=\textheight,keepaspectratio}
    \begin{tabular}{|c|c|c|c|c|c|c|c|c|c|c|} 
   Class &  $\theta_0$ &
      $\theta_1 $ & $\theta_2$  &   $\theta_3$ & $\theta_4$ & $\theta_5$ & $\chi_1(k,l)$ & $\chi_2(k)$ & $\chi_3(k,l)$ & $\chi_4(k,l)$ \\
      \hline
      
$A_1$ & 1 & $q(q+1)^2/2$ & $q(q^2+1)/2$ & $q(q^2+1)/2$ & $q^4$ & $q(q-1)^2/2$  & $(q+1)^2(q^2+1)$ & $q^4-1$ & $q^4-1$ & $(q-1)^2(q^2+1)$  \\

$A_2$ & 1& $q(q+1)/2$ & $-q(q-1)/2$ &  $q(q+1)/2$ & & $-q(q-1)/2$ & $(q+1)^2$ & $q^2-1$ & $-q^2-1$ & $(q-1)^2$  \\

   $A_{31}$ & 1&$q(q+1)/2$ & $q(q+1)/2$ & $-q(q-1)/2$ & & $-q(q-1)/2$ & $(q+1)^2$ & $-q^2-1$ & $q^2-1$ & $(q-1)^2$\\         
            
 $A_{32}$ & 1&$q/2$ & $q/2$ &$q/2$ & & $q/2$ & $2q+1$ & $-1$ & $-1$ & $-2q+1$ \\
   
  $A_{41}$ & 1&$q/2$ & $-q/2$ & $-q/2$ & & $q/2$  & 1 & $-1$ & $-1 $ & 1  \\
   
   $A_{42}$ & 1&$-q/2$ & $q/2$ & $q/2$ & & $-q/2$& 1 & $-1$ & $-1 $ & 1  \\  
   
  & & & & & & & & & &  \\

$B_1(i,j)$ & 1& $2$ & 1& 1 & 1  & & $\alpha_{ik}\alpha_{jl}+\alpha_{il}\alpha_{jk}$ &  & &  \\

$B_2(i)$ & 1 & & 1& $-1 $ & $-1$ & & & $-\theta_{ik}-\theta_{qik}$ & &  \\

$B_3(i, j)$ & 1& & $-1$& $1$ & $-1$ & & & & $-\alpha_{ik}\beta_{jl}$  &  \\

$B_4(i,j)$ & 1& & $-1$& $-1$ & 1 & $-2$ & & & & $\beta_{ik}\beta_{jl} + \beta_{il}\beta_{jk}$\\

$B_5(i)$ & 1& $-1$ & & & $1$ & $1$  & & & & \\

& & & & & & & & & &  \\ 

$C_1(i)$ & 1& $q+1$ & $q$ & $1$ & $q$ & & $(q+1)(\alpha_{ik}+\alpha_{il})$ & & $(q-1)\alpha_{ik}$ &  \\

$C_2(i)$ & 1& $q+1$& $1$ & $q$ & $q$ &  & $(q+1)\alpha_{ik}\alpha_{il}$ & $(q-1)\alpha_{ik}$ & &  \\

$C_3(i)$ & 1& & $-1$ &  $q$ & $-q$ & $q-1$ & & & $-(q+1)\beta_{il}$ & $-(q-1)(\beta_{ik}+\beta_{il})$\\

$C_4(i)$ & 1& & $q$ &  $-1$ & $-q$ & $q-1$ & & $-(q+1)\beta_{ik}$ & & $-(q-1)\beta_{ik}\beta_{il}$\\

& & & & & & & & & &  \\ 

$D_1(i)$ & 1& 1& & 1 & & & $\alpha_{ik}+\alpha_{il}$ & & $-\alpha_{ik}$ &  \\

$D_2(i)$ & 1& 1& 1 &  & & & $\alpha_{ik}\alpha_{il}$ & $-\alpha_{ik}$ & &  \\

$D_3(i)$ & 1& & $-1$ & & & $-1$ & & & $-\beta_{il}$ & $\beta_{ik}  + \beta_{il}$\\

$D_4(i)$ & 1& & & $-1$ & & $-1$ & & $-\beta_{ik}$ & & $\beta_{ik}\beta_{il}$  \\
  
    \end{tabular}
    \end{adjustbox}
  \end{center}
\end{table}

 \tabcolsep=0.11cm
\begin{table}[h]
  \begin{center}
    \caption{character table for $\Sp(4, q)$, $q$ even, continued; $\tau_i = \zeta_{q^2+1}^i+\zeta_{q^2+1}^{-i}$}
    \label{tab:table9}
\begin{adjustbox}{max width=\textwidth,max totalheight=\textheight,keepaspectratio}
    \begin{tabular}{|c|c|c|c|c|c|c|c|c|c|c} 
   Class &  $\chi_5(k)$ &
      $\chi_6(k) $ & $\chi_7(k)$  &   $\chi_8(k)$ & $\chi_9(k)$ & $\chi_{10}(k)$ & $\chi_{11}(k)$ & $\chi_{12}(k)$ & $\chi_{13}(k)$ &  \\
      \hline
      
$A_1$ & $(q^2-1)^2$ & $(q+1)(q^2+1)$ & $(q+1)(q^2+1)$ & $(q-1)(q^2+1)$ & $(q-1)(q^2+1)$ & $q(q+1)(q^2+1)$  & $q(q+1)(q^2+1)$ & $q(q-1)(q^2+1)$ & $q(q-1)(q^2+1)$   \\

$A_2$ & $-q^2+1 $& $q+1$ & $q^2+q+1$ &  $q-1$ &  $-q^2+q-1$ & $q(q+1)$ & $q$ & $q(q-1)$ & $-q$   \\

   $A_{31}$ & $-q^2+1 $& $q^2+q+1$ & $q+1$ &  $-q^2+q-1$ &  $q-1$ & $q$ & $q(q+1)$ & $-q$ & $q(q-1)$   \\        
            
 $A_{32}$ & 1& $q+1$ & $q+1$ &$q-1$ & $q-1$ & $q$ & $q$ & $-q$ & $-q$  \\
   
  $A_{41}$ & 1&$1$ & $1$ & $-1$ & $-1$ &   &  &  &    \\
   
   $A_{42}$ & 1&$1$ & $1$ & $-1$ & $-1$ &   &  &  &    \\
   
  & & & & & & & & & &  \\

$B_1(i,j)$ &  & $\alpha_{ik}\alpha_{jk}$ & $\alpha_{ik}+\alpha_{jk}$&  &   & $\alpha_{ik}\alpha_{jk}$ & $\alpha_{ik}+\alpha_{jk}$ & &   \\

$B_2(i)$ &  & $\alpha_{ik}$ &  & $-\beta_{ik} $ &  & $-\alpha_{ik}$ & & $-\beta_{ik}$ &  \\

$B_3(i, j)$ & & & $\alpha_{ik}$&  & $-\beta_{jk}$ & & $-\alpha_{ik}$ & & $-\beta_{jk}$    \\

$B_4(i,j)$ & & & & $-\beta_{ik}\beta_{jk}$ & $-\beta_{ik}-\beta_{jk}$&  & & $\beta_{ik}\beta_{jk}$ & $\beta_{ik}+\beta_{jk}$ \\

$B_5(i)$ & $\tau_{ik}+\tau_{qik}$&   & & &    & & & & \\

& & & & & & & & & &  \\ 

$C_1(i)$ & & $(q+1)\alpha_{ik}$ & $q+1+\alpha_{ik}$ &  & $q-1$ &   $(q+1)\alpha_{ik}$ & $q+1+q\alpha_{ik}$  &  & $q-1$  \\

$C_2(i)$ & &   $q+1+\alpha_{2ik}$& $(q+1)\alpha_{ik}$ &   $q-1$ &  &   $q+1+q\alpha_{2ik}$& $(q+1)\alpha_{ik}$    & $q-1$ & \\

$C_3(i)$ & & & $q+1$ &  $(q-1)\beta_{ik}$ & $q-1-\beta_{ik}$ & & $-q-1$ &   $-(q-1)\beta_{ik}$ & $-q+1-q\beta_{ik}$\\

$C_4(i)$ & &  $q+1$ & &   $q-1-\beta_{2ik}$ &$(q-1)\beta_{ik}$ &  $-q-1$ & & $-q+1-q\beta_{2ik}$ &  $-(q-1)\beta_{ik}$ \\

& & & & & & & & & &  \\ 

$D_1(i)$ & & $\alpha_{ik}$ & $1+\alpha_{ik}$ &  & $-1$ & $\alpha_{ik}$ & $1$ & & $-1$   \\

$D_2(i)$ & &  $1+\alpha_{2ik}$ & $\alpha_{ik}$ & $-1$  & & $1$ & $\alpha_{ik}$ & $-1$ &   \\

$D_3(i)$ & & & $1$ & $-\beta_{ik}$ & $-1-\beta_{ik}$ &  & $-1$ & $\beta_{ik}$ & 1 \\

$D_4(i)$ & & $1$ & & $-1-\beta_{2ik}$ & $-\beta_{ik}$ & $-1$ & & $1$ &  $\beta_{ik}$  \\
  
    \end{tabular}
    \end{adjustbox}
  \end{center}
  
\end{table}


\clearpage






\begin{thebibliography}{widest entry}
\bibitem[Da2007]{dabbaghianabdoly2007} Dabbaghian-Abdoly, Vahid. ``Characters of some finite groups of Lie type with a restriction containing a linear character once." Journal of Algebra 309.2 (2007): 543-558.
\bibitem[En1972]{En1972} Enomoto, Hikoe. ``The characters of the finite symplectic group $Sp (4, q), q= 2^f$." Osaka J. Math 9.1 (1972): 75-94.
 \bibitem[Iw1965]{Iw1965} Iwahori, Nagayoshi, and Hideya Matsumoto. ``On some Bruhat decomposition and the structure of the Hecke rings of $ p $-adic Chevalley groups." Publications Mathématiques de l'IHÉS 25 (1965): 5-48.
\bibitem[J2022]{J2022} Johnson-Leung, Jennifer, Brooks Roberts, and Ralf Schmidt. Stable Klingen vectors and Paramodular Newforms. Springer, 2023.
 \bibitem[L2010]{L2010} Lust, Jaime. Verifying depth-zero supercuspidal L-packets for inner forms of $GSp(4)$. Diss. UC San Diego, 2010.
 \bibitem[Ro2006]{Ro2006} Roberts, Brooks, and Ralf Schmidt. ``A decomposition of the spaces $S_k (\Gamma_0 (N))$ in degree 2 and the construction of hypercuspidal modular forms." Proceedings of the 9th Autumn Workshop on Number Theory, Hakuba, Japan, 2006.
 \bibitem[Ro2007]{Ro2007} Roberts, Brooks, and Ralf Schmidt. Local newforms for $GSp (4)$. Vol. 1918. Springer Science and Business Media, 2007.
 \bibitem[Ros2016]{Ros2016} R\"{o}sner, Mirko. Parahoric restriction for {${\rm GSp}(4)$} and the inner cohomology of {S}iegel modular threefolds. Heidelberg University PhD Thesis, 2016. Available at\\ {\tt https://www.mathi.uni-heidelberg.de/fg-sga}
 \bibitem[Ros2018]{Ros2018} R\"{o}sner, Mirko. ``Parahoric restriction for $\GSp(4)$." Algebras and Representation Theory 21 (2018): 145-161.
 \bibitem[RSY2022]{RSY2022} Roy, Manami, Ralf Schmidt, and Shaoyun Yi. ``Dimension formulas for Siegel modular forms of level 4." Mathematika 69.3 (2023): 795-840.
 \bibitem[Sa1993]{Sa1993} Sally, Paul, and Marko Tadic. ``Induced representations and classifications for $GSp (2, F)$ and $Sp (2, F)$." M\'{e}m. Soc. Math. France (NS) 52 (1993): 75-133.
 \bibitem[Sh1982]{Sh1982} Shinoda, Ken-ichi. ``The characters of the finite conformal symplectic group, $CSp (4, q)$." Communications in Algebra 10.13 (1982): 1369-1419.
 \bibitem[Yi2021]{Yi2021} Yi, Shaoyun. ``Klingen $\mathfrak{p}^2$ vectors for GSp(4)." The Ramanujan Journal 54, no. 3 (2021): 511-554.
 
 
\end{thebibliography}
\end{document}